\pdfoutput=1

\documentclass[11pt, oneside]{article}   	% use "amsart" instead of "article" for AMSLaTeX format
\usepackage{geometry}                		% See geometry.pdf to learn the layout options. There are lots.
\usepackage[parfill]{parskip}    			% Activate to begin paragraphs with an empty line rather than an indent
\usepackage{graphicx}				% Use pdf, png, jpg, or epsÂ§ with pdflatex; use eps in DVI mode
\usepackage{amssymb}
\usepackage{mathtools}
\usepackage{enumerate}
\usepackage{amsmath}

\usepackage[utf8]{inputenc}
\usepackage[english]{babel}
\usepackage[scr]{rsfso}
\usepackage{mathtools}
\usepackage{amsthm}

\usepackage{natbib}

\usepackage{hyperref}

\newtheorem{lem}{Lemma}
\newtheorem{cor}{Corollary}
\newtheorem{thm}{Theorem}

\newtheorem*{prop1.1}{Proposition 1.1 of \citep{stan73}}

\newtheorem*{thm1.2}{Theorem 1.2 of \citep{stan73}}

\theoremstyle{definition}
\newtheorem{defi}{Definition}
\newtheorem{eg}{Example}
\title{Chromatic Polynomial and Heaps of Pieces }
\author{Bishal Deb\footnote{This work was done by the author during his undergraduate stay in Chennai Mathematical Institute, H1, SIPCOT IT Park, Siruseri Kelambakkam, India - 603103.} - \texttt{bishal@gonitsora.com}\\Laboratoire d’Informatique Gaspard Monge (LIGM), \\University of Paris-Est Marne-la-Vallée (UPEM),\\5 Boulevard Descartes, 77420 Champs-sur-Marne, France. }
\date{\today}							% Activate to display a given date or no date

\begin{document}

\maketitle

\begin{abstract}

Stanley in his paper [Stanley, Richard P.: Acyclic orientations of graphs In: Discrete Mathematics 5 (1973), Nr. 2, S. 171–178.] provided interpretations of the chromatic polynomial when it is substituted with negative integers. Greene and Zaslavsky interpreted the coefficients of the chromatic polynomial in [Greene, Curtis ; Zaslavsky, Thomas: On the interpretation of Whitney numbers through arrangements of hyperplanes, zonotopes, non-Radon partitions, and orientations of graphs. In: Transactions of the American Mathematical Society 280 (1983), jan, Nr. 1, S. 97–97.]. We shall develop an involution on factorisations of heaps of pieces and using this involution, we shall provide bijective proofs to results from both the papers.
\end{abstract}

\section{Introduction}

Viennot in \citep{xv86} developed the theory of heaps of pieces based on the theory of commutation monoids developed by Cartier and Foata in \citep{Cartier1969}. This theory has been used to solve several problems in combinatorics and establish several bijections. In this paper, we shall develop an involution on factorisations of heaps. We use it to provide new bijective proofs to results involving chromatic polynomials of graphs from \citep{stan73}. We then interpret the coefficients of chromatic polynomials in terms of heaps of pieces and establish bijections developed in \citep{Greene1983}.

A succint introduction to heaps of pieces for our purpose is provided in Section~\ref{intro}. We shall then fix some notation in Section~\ref{notation}. In Section~\ref{racks} we shall look at the fundamental lemma using which we provide new bijective proofs to several theorems on chromatic polynomials. We shall postpone the proof of the results to Appendix as it only involves case checking. In Section~\ref{stan73}, we shall look at the proofs of reciprocity theorems from \citep{stan73}. In Section~\ref{coefficients} we shall provide interpretations to coefficients of chromatic polynomial in terms of heaps of pieces.

\section{An Introduction to Heaps of Pieces}\label{intro}

This section is a short introduction to the theory of Heaps of Pieces, which were first introduced in \citep{xv86}. For a more detailed survey, refer to the course notes in the format of a video book on Heaps of Pieces at \citep{heaps}.

We begin with a set $P$, we call it the set of basic pieces and each element of this set as a piece. We have a relation $\mathscr{C}$ on this set that is reflexive and symmetric i.e., for $a,b\in P$, $a\mathscr{C}a$ and if $a\mathscr{C}b$ then $b\mathscr{C}a$. We call $\mathscr{C}$ as the \emph{concurrency} or \emph{dependency} relation, and if $a\mathscr{C}b$ for some $a,b\in P$, we say that $a$ is dependent on $b$.\\
 
\begin{defi}[Poset definition of Heaps]\label{poset}
A \textit{heap} on $P$ is the tuple $((E,\leq), \epsilon)$ where $(E,\leq)$ is a finite poset and $\epsilon:E\rightarrow P$ such that 

\begin{itemize}
\item[(i)] For every $a,b\in E$ such that $\epsilon(a) \mathscr{C}\epsilon(b)$ either, $a\leq b$ or $b\leq a$ i.e., $a$ and $b$ are comparable.

\item[(ii)] For every $a,b \in E$ such that $b$ covers $a$, $\epsilon(a) \mathscr{C}\epsilon(b)$.
\end{itemize}

\end{defi}

We call the elements of $E$ as the \emph{pieces} of $E$. When $a\leq b$ we say that $a$ \emph{is below} $b$ or $b$ \emph{is above} $a$. We call $\epsilon$ to be the \emph{projection map}.

We often refer to $E$ as the heap when the order relation and $\epsilon$ is understood.\\

\begin{eg}Let $G  = (V,E)$ be a graph. We consider $V$ to be the set of basic pieces with dependency relation given by edge relations and each vertex is dependant on itself. The Hasse diagram of a heap on the path graph with four vertices $P_4$ is given in Figure~\ref{first figure}. \\
\end{eg}

\begin{figure}[h]
\begin{center}
\includegraphics[scale=0.35]{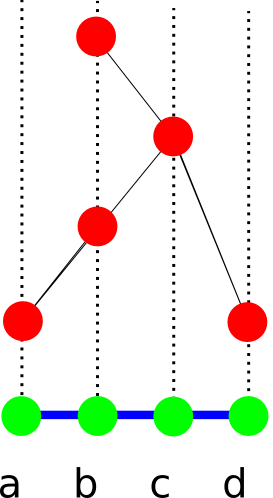}
\caption{Heap on $P_4$. Here $\{a,b,c,d\}$ is the set of vertices of the graph.}\label{first figure}
\end{center}
\end{figure}

\begin{defi}[Subheap] A \textit{subheap} of a heap $((E,\leq) , \epsilon)$ is $((F,\leq'),\epsilon')$ where $(F,\leq')$ is an induced subposet of $(E,\leq)$, and $\epsilon'$ is the restriction of $\epsilon$ to $F$.\\
 
\end{defi}

\begin{defi}[Multiplication of two heaps] For heaps $((E,\leq) , \epsilon)$ and $((F,\leq'),\epsilon')$ on $P$, we define $((E,\leq) , \epsilon)\odot ((F,\leq'),\epsilon')$ as the heap $((H, \leq''),\epsilon'')$ where 

\begin{itemize}
\item[--] $H$ is the disjoint union of $E$ and $F$.
\item[--] $\epsilon''$ is the map such that restricted to $E$ and $F$, it is $\epsilon$ and $\epsilon'$ respectively.
\item[--] $\leq''$ is the transitive closure of the following relations

\begin{enumerate}[(i)]
\item $a,b\in E$ and $a\leq b$,

\item $a,b\in F$ and $a\leq'b$,
\item $a\in E,b\in F$ and $\epsilon(a)\mathscr{C}\epsilon(b)$.
\end{enumerate}
\end{itemize}
\end{defi}

When the order relations and projection maps are understood we then just denote the product of $E$ and $F$ by $E\odot F$.

The product $\odot$ is associative and thus the set of all heaps on $P$ forms a monoid with generators, the heaps with just one piece. We use $\mathcal{H}(P,\mathscr{C})$ to denote this monoid.

We can also view a heap geometrically as:\\

\begin{defi}[Geometric definition of Heaps]\label{geometric} A \textit{heap} is a finite subset $E\subset P\times\mathbb{N}$ such that the following conditions hold:

\begin{enumerate}[(i)]
\item For $(a,i),(b,j)\in E$, if $a\mathscr{C}b$ then $i\neq j$.

\item For $(a,i)\in E$ if $i>0$, then there exists $ b\in P$ such that $a\mathscr{C}b$ and $(b,i-1)\in P$.
\end{enumerate}

\end{defi}

All elements with second entry $i\in \mathbb{N}$ are said to be in level $i$.\\

\begin{eg} Figure~\ref{levels} is the illustration of Example~1 with levels.\\\end{eg}

\begin{figure}[h]
\begin{center}
\includegraphics[scale=0.30]{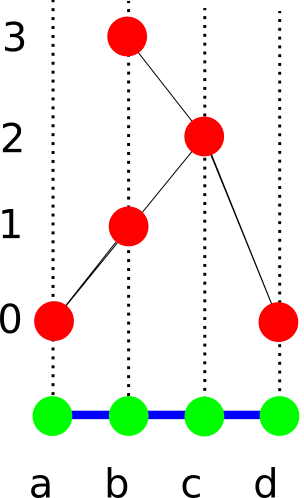}
\caption{Heap on $P_4$.}\label{levels}
\end{center}
\end{figure}

Definitions~\ref{poset} and \ref{geometric} are equivalent as sets: given a heap $((E,\leq),\epsilon)$ of Definition~\ref{poset}, we take all the minimal pieces in $((E,\leq),\epsilon)$ and put them at level $0$. Thus, we get that $E = T_0\odot E_1$ where $T_0$ is the set of minimal pieces of $E$. We put the minimal pieces of $E_1$ in level 1 and so on. Given a heap of Definition~\ref{geometric}, we take $\epsilon$ to be the projection of pieces to $P$. For pieces $(a,i),(b,j)$ with $i<j$ and $a\mathscr{C}b$ we take $(a,i)< (b,j)$.\\

\begin{defi}[Cartier-Foata Monoids or Commutation Monoids] Let $P^*$ be the free monoid generated by $P$. Let $C$ be a relation on $P$ which is symmetric and irreflexive. Let $\equiv_C$ denote the commutation relation on $P^*$ generated by the commutations $ab=ba$ iff $aCb$. Then $P^*/\equiv_C$ is a \textit{Cartier-Foata monoid} or a \textit{commutation monoid}.

\end{defi}

Notice that for a set of basic pieces $P$, with dependency relation $\mathscr{C}$, its complement $C = \mathscr{C}^c$ is a commutation relation. We state the following theorem without proof (See Proposition~3.4 in \citep{xv86}):\\

\begin{thm}
The map $\phi: P^*/\equiv_C\to \mathscr{H}(P,\mathscr{C})$ sending $a_1\ldots a_n\to a_1\odot\ldots\odot a_n$ is an isomorphism of monoids.
\end{thm}

Thus, each heap on $P$ can be represented as words in a commutation class of $P^*/\equiv_C$. We mention two ways of representing heaps as words.\\

\begin{lem}[Cartier-Foata Normal Form]

Given a heap $E$ it can be represented as product of blocks $[w_0][w_1]\ldots[w_k]$ where each $w_i\in P^*$ and the following hold:

\begin{itemize}
\item The letters in $w_i$ commute pairwise.
\item For each letter $a$ in $w_i$ there is a letter $b$ in $w_{i+1}$ such that $ab\neq ba$.
\end{itemize}
This representation of blocks is unique upto commutations of the letters in $w_i$.
\end{lem}

For a proof of the above lemma see Corollary~3.5 in \citep{xv86}. It is due to Cartier and Foata proved in \citep{Cartier1969}. Note that $i^{th}$ block in the Cartier-Foata normal form of $E$ denotes the elements in $i$ level of $E$ as per the geometric representation of heap.\\

\begin{eg} In Example~1 the Cartier-Foata normal form of the heap is $[ad][b][c][b]$.\\\end{eg}

\begin{lem}[Knuth Normal Form or Lexicographic Normal Form]

Given a heap $E$ on $P$ and a total ordering on $P$ it can be representated uniquely as $a_1\ldots a_n$ such that $a_1$ is the smallest minimal element of $E$ which gives $E = a_1\odot E_1$, $a_2$ is the smallest minimal element of $E_1$ and so on.\\

\end{lem}

The above lemma was proved by Anisimov and Knuth in \citep{Anisimov1979} in the context of Commutation Monoids.\\

\begin{eg} In Example~2 if we take the total ordering as $a<b<c<d$, then the lexicographic normal form of the heap is $abdcb$.\\\end{eg}

\begin{defi}[Trivial Heaps] A non-empty heap $E$ is called \textit{trivial} when each piece of the heap commutes with every other piece of the heap or in other words, all pieces of the heap lie at level 0.\\

\end{defi}

\begin{defi}[Multilinear Heaps] A heap $E$ is called \textit{multilinear} if $\epsilon:E\to P$ is a bijection where $P$ is the set of basic pieces.\\

\end{defi}

\begin{eg}\label{multilinear} The heap in Figure~\ref{multi} is an example of a multilinear heap.\\\end{eg}

\begin{figure}[h]
\begin{center}
\includegraphics[scale=0.35]{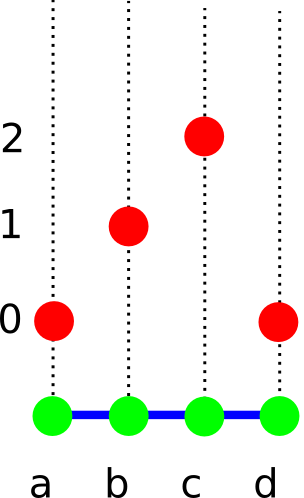}
\end{center}

\caption{A Multilinear Heap on the graph $P_4$.}\label{multi}
\end{figure}

\begin{defi}[Antipyramids] A heap is called an \textit{antipyramid} if it has a unique piece in lowermost level.

\end{defi}

\section{Notation and Terminology}\label{notation}

Our graph will be denoted by $G=(V,E)$ where $n$ is the number of vertices and $m$ is the number of edges. We call a decomposition of $F$ into subheaps $F_0,\ldots, F_{k-1}$ such that $F = F_0\odot\ldots\odot F_{k-1}$ a \emph{factorisation} or a \emph{layer factorisation} of heap $F$. The heaps $F_0,\ldots, F_{k-1}$ are called the \emph{factors} or \emph{layers} of heap $F$. A trivial layer factorisation of heap $F$ is a factorisation in which each of the factors are trivial heaps. We also refer to a trivial layer factorisation of a heap as  a \emph{rack} of the heap. 

Given a total order on the set of basic pieces, we can order the pieces of $F$ in the order in which they are written in the lexicographic normal form when read from left to right. We call this order the \emph{lexicographic order} of the heap. The lexicographic order is a linear extension of the poset order of the heap. There is a rack such that its layer $i-1$ has only the $i^{th}$ piece of $F$. We call it the \emph{lexicographic rack}. For a rack $T$ let $\#(T)$ denote the number of layers of $T$. For example, the lexicographic rack of the heap of Example~1 is $a\odot b\odot d\odot c\odot b$.

In a rack we call a piece \emph{lonely} if it is the only piece in its layer. For a heap $F$ let  $\beta_F(k)$ denote the number of racks on $F$ with $k$ layers. Let $\mathfrak{b}_F(k)$ denote the number of layer factorisations of $F$ with $k$ layers. When $F$ is understood we may choose to ignore the subscript $F$. 

\section{The Heaps and Racks lemma}\label{racks}

We now come to the most important part of this paper. Consider the following algorithm on the racks of a fixed heap $F$:

\begin{enumerate}
\item We first number the pieces in the rack $T$ from $0$ according to the lexicographic order.

\item We define the \emph{Transfer Set} to be the set consisting of two kinds of pieces:
\begin{enumerate}
\item Pieces which are not lonely.
\item Lonely pieces whose number on them due to Step 1. do not match the number of its layer.
\end{enumerate}
If the transfer set is non-empty we call its smallest piece the \emph{Transfer Piece}.

\item If the Transfer Set is empty we return $T$ as the output of the algorithm.

\item If the transfer piece is lonely then we put it in the layer below it. If it is not lonely, we create a new layer with the transfer piece as the only piece just above its old layer.\\
\end{enumerate}

We call this algorithm the \emph{heaps and racks involution}.

\textbf{Example.} The racks in Figure~4 is an example of the heaps and racks involution applied on two racks on the heap of Figure~1. When the algorithm is applied on the left rack we get the right rack and when the algorithm is applied to the right rack we get the left rack.

\begin{figure}[h]
\begin{center}
\includegraphics[scale=0.30]{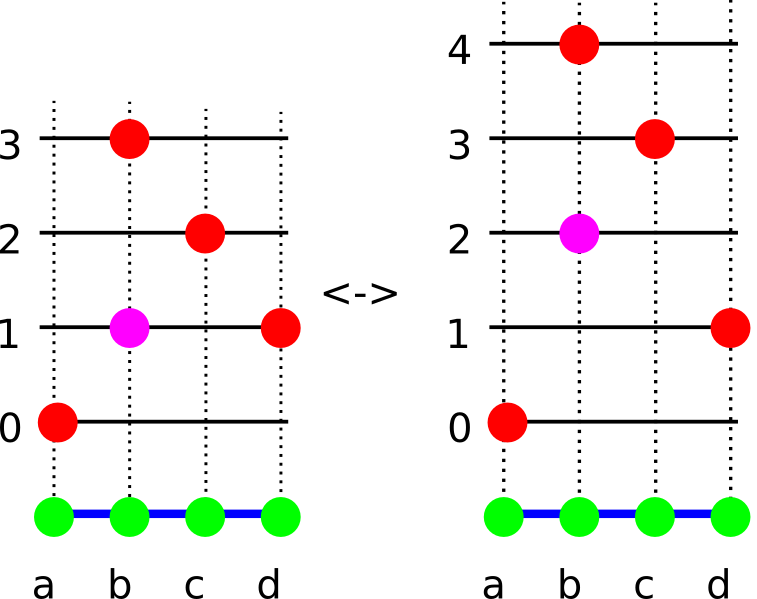}
\end{center}
\caption{Heaps and Racks Involution applied to two rack. Here the order on the pieces are $a<b<c<d$. The purple piece in the heap denotes the transfer piece.}
\end{figure}

\begin{lem}\label{algo} The result of the above algorithm is a rack of $F$. 
\end{lem}

We prove the lemma in the Appendix.

Let $R(F)$ be the set of racks of $F$. Then the algorithm that we described above gives a function $f: R(F) \rightarrow R(F)$.\\

\begin{lem}{(Heaps and Racks Lemma)} Let $f: R(F) \rightarrow R(F)$ be the algorithm above performed on racks of $F$. Then, $f$ is an involution whose only fixed point is the lexicographic rack of $F$. Further, $|\#(T) - \#(f(T))|\leq 1$ with $\#(T) = \#(f(T))$ if and only if $T$ is the lexicographic rack.

\end{lem}

From the description of the algorithm, it is clear that for $T \in R(F)$, $\#(T) = \#(f(T))$ if and only if $T$ is a fixed point of $f$. A rack $T$ is a fixed point if and only if all pieces are lonely and are in the same layer as its order. Thus, $T$ has to be the lexicographic rack.

If $T$ is not the lexicographic rack, then from the algorithm $f(T)$ will have one more layer than $T$ if the transfer piece is not lonely and one less layer than $T$ if it is lonely. Thus we get that $|\#(T) - \#(f(T))|\leq 1$ with $\#(T) = \#(f(T))$ if and only if $T$ is the lexicographic rack.

We postpone the remaining part of the proof of the lemma to the Appendix where we show that $f$ is an involution.\\

\begin{cor}

The following identity holds true for any fixed heap $F$,

\begin{equation}\label{inv}
\sum_{k\geq 1} (-1)^{k}\beta(k) = (-1)^{|F|}.
\end{equation}

\end{cor}

Note that this corollary is universal in the sense that it holds for all heaps. But in this paper we shall only be interested in multilinear heaps with the basic pieces as graph vertices.\\

\begin{cor}\label{lowerspecial} If the heaps and racks involution is applied to a rack with the largest piece in the bottom layer then the rack obtained has the largest piece in the bottom layer. Here the ordering is the inherent ordering on the vertices.

\end{cor}

\section{Stanley's Reciprocity Theorems}\label{stan73}

In this section we shall provide proofs to several results from \citep{stan73}. Several of these results were proved in \citep{heaps}. We provide new proofs to most of the results and we mention wherever the proof is not due to the author. We use the results from the previous section to obtain bijective proofs of the results.

\subsection{Orientation of Graphs and Multilinear Heaps}

The proof of the following lemma was discussed in the course by Viennot and is present in Chapter 5a of \citep{heaps}.\\

\begin{lem}\label{orientations and multilinear heaps} Let $G = (V,E)$ be an undirected simple graph. There is a bijection between orientations of G and multilinear heaps on $G$.

\end{lem}

\begin{proof} Given a multilinear heap $H$ on graph $G$, let $W$ be its lexicographic word. Then, if there is an edge between vertices $u$ and $v$ we orient it from $u$ to $v$ if $u$ occurs to the left of $v$ in $W$, else we orient it from $v$ to $u$.

Clearly, this orientation that we get is acyclic. Thus, we get a map $\phi: \mathscr{M}(G)\to \mathscr{A}(G)$ where $\mathscr{M}(G)$ denotes the set of multilinear heaps on $G$ and $\mathscr{A}(G)$ denotes the set of acyclic orientations on $G$.

Now we give an algorithm to get the lexicographic word of a multilinear heap from an acyclic orientation. 

Begin with the empty word. At each step keep removing the source with the lowest order and concatenate the corresponding letter to the right of the word. Stop when there are no more vertices left in the graph.

The word gives a multilinear heap as there is exactly one letter for each vertex in the word formed. Let us call this map $\psi: \mathscr{A}(G)\to \mathscr{M}(G)$.

It is not difficult to see that $\phi$ and $\psi$ are inverses of each other. \end{proof}

%Thus, we get a proof to Proposition 1.1 in \citep{stan73}.

\subsection{Graph Colouring and Racks of a Multilinear heap}

A proper colouring of a graph partitions the vertex set into independent sets. Each independent set corresponds to a trivial heap. Thus, if we fix an ordering on the colours we get a rack of a multilinear heap. Hence, colouring a graph with $\lambda$ colours is same as first picking a rack on graph of size $k\leq \lambda$ and then picking $k$ colours from the set of colours $\{1,\ldots, \lambda\}$.

Thus, we get an expression for the chromatic polynomial $\mathcal{X}_G(\lambda)$

\begin{equation}\label{chigeneral}
\mathcal{X}_G(\lambda) = \sum_{F \in\mathscr{M}(G)}\sum_{k\geq 0}\beta_F(k)\frac{\lambda(\lambda -1 )\ldots(\lambda - k+1)}{k!}.
\end{equation} This also gives, 

\begin{equation}\label{chi}\mathcal{X}_G(-1) = \sum_{F \in\mathscr{M}(G)}\sum_{k\geq 0}(-1)^k\beta_F(k).
\end{equation}

\noindent{\textbf{Corollary 1.3 from \citep{stan73}.} \emph{If $G$ is a graph with $n$ vertices, then $(-1)^n \mathcal{X}_G(-1)$ is equal to the number of acyclic orientations of $G$.}}

\begin{proof}
We want to show that $$\mathcal{X}_G(-1) = (-1)^{n} |\mathscr{A}(G)|.$$ But this is immediate from Equations~(\ref{inv}),~(\ref{chi}) and the bijection between acyclic orientations and multilinear heaps in Lemma~\ref{orientations and multilinear heaps}.\end{proof}

\subsection{Theorem 1.2 from \citep{stan73}}

We first recall Proposition 1.1  and the definition of $\bar{\mathcal{X}}(\lambda)$ from \citep{stan73}.

\vspace{2mm}

\begin{prop1.1}
$\mathcal{X}(\lambda)$ is equal to the number of pairs $(\sigma, \mathcal{O})$ where $\sigma$ is any map $\sigma: V\rightarrow \{1,\ldots, \lambda\}$ and $\mathcal{O}$ is an orientation subject to the two conditions:

\begin{enumerate}
\item[(a)] The orientation $\mathcal{O}$ is acyclic.
\item[(b)] If $u\rightarrow v$ is in the orientation, $\sigma(u)>\sigma(v)$.
\end{enumerate}
\end{prop1.1}

Let $\bar{\mathcal{X}}(\lambda)$ be the number of pairs as in the previous proposition with $>$ replaced with $\geq$ in \emph{(b)}.

We have the following lemma:

\begin{lem}\label{bijpairs}
\begin{enumerate}
\item[(a)] There is a bijection between pairs $(\sigma,\mathcal{O})$ in Proposition~1.1 and pairs $(F, (T, \mathcal{S}))$ where $F$ is a multilinear heap on $G$ and $T$ is a rack on $F$ and $ S\in\displaystyle\binom{[\lambda]}{|T|}$. Here,  $\displaystyle\binom{[\lambda]}{|T|}$ denotes the collection of subsets of $\{1,\ldots,\lambda\}$ of size $|T|$.

\item[(b)] There is a bijection between pairs $(\sigma,\mathcal{O})$ and pairs $(F, (\mathcal{L}, \mathcal{S}))$ where $F$ is a multilinear heap on $G$ and $\mathcal{L}$ is a layer factorisation on $F$ and $ S\in\displaystyle\binom{[\lambda]}{|T|}$ and $\mathcal{O}$ is an acyclic orientation and $\sigma:V\to \{1,\ldots,\lambda\}$ such that if $u\rightarrow v$ is in the orientation, $\sigma(u)\geq\sigma(v)$.
\end{enumerate}

\end{lem}

\begin{proof}
From Lemma~\ref{orientations and multilinear heaps} we have a bijection between $\mathcal{O}$'s and $F$'s. Note that $\sigma^{-1}(i)$ is either empty or forms an independent set in case \emph{(a)} and a  heap of $V$ in case \emph{(b)}. Thus, $\sigma^{-1}(i)$ is a layer in either case. The bijection is established.\end{proof}

From Lemma~\ref{bijpairs} we get that, \begin{equation}\label{chibar}
\bar{\mathcal{X}}(\lambda) = \sum_{F \in \mathscr{M}(G)}\sum_{j\geq 0}\mathfrak{b}_F(j)\binom{\lambda}{j}.
\end{equation}

\begin{thm1.2} For all non-negative integers $\lambda$,
$$\bar{\mathcal{X}}(\lambda) = (-1)^n\mathcal{X}(-\lambda).$$
\end{thm1.2}

\begin{proof} We expand both sides to see that we need to prove $$\sum_{F \in \mathscr{M}(G)}\sum_{j\geq 0}\mathfrak{b}_F(j)\binom{\lambda}{j} = (-1)^{n}\sum_{F \in \mathscr{M}(G)}\sum_{k\geq 0}(-1)^k \beta_F(k)\binom{\lambda+k-1}{k}.$$ 

We get the left hand side from Equation~(\ref{chibar}) and the right hand side from Equation~(\ref{chigeneral}). The left hand side counts the number of coloured layer factorisations of $F$ with colours in $[\lambda]$.

We use the identity $$\binom{\lambda+k-1}{k}=\sum_{i=0}^{k}\binom{k}{i}\binom{\lambda-1}{i}$$ and further expand the right hand side to get, $$\sum_{F \in \mathscr{M}(G)}\sum_{k\geq 0}\sum_{i=0}^{k}(-1)^k \beta_F(k)\binom{k}{i}\binom{\lambda-1}{i}.$$

We can interpret the term $\beta_F(k)\binom{k}{i}\binom{\lambda-1}{i}$ as first choosing a rack of $F$ into $k$ layers $T_0,\ldots,T_{k-1}$, then choosing layers $T_{l_1},\ldots,T_{l_i}$ with $0\leq l_1<\ldots<l_i\leq k-1$ and colours $1\leq c_1<\ldots<c_i\leq \lambda-1$. Now we assign the colour $c_1$ to the layers $T_0,\ldots,T_{l_1}$, and for $j>1$, $c_j$ to the layers $T_{l_{j-1}+1},\ldots,T_{l_j}$. Finally, we assign the colour $\lambda$ to the remaining top layers (if there are any).

Thus we get that $$F=E_0\odot \ldots \odot E_{i}$$ where we take $$E_0=T_0\odot\ldots \odot T_{l_1}$$ for $0<j<i$, $$E_j=T_{l_{j}+1}\odot\ldots\odot T_{l_{j+1}}$$ and finally $$E_{i}=T_{l_{i}+1}\odot\ldots\odot T_{k}$$ where $E_j$ is assigned the colour $c_j$ for $j<i$ and $\lambda$ is assigned to $E_{i}$.

Basically, the layers $E_j$ obtained are the product of the trivial layers with the same colour. We then assign the obtained layer the original colour assigned to each of the trivial layers.

Let $\mathfrak{E}_F$ denote set of the coloured layer factorisations of $F$ with colours in $[\lambda]$. Here $[\lambda] = \{1,\ldots, \lambda\}$

We call a rack $(T_0,\ldots,T_{k-1})$ weakly coloured with colours in $[\lambda]$ if there is a map $f:~\{0,\ldots,k-1\}\to[\lambda]$ such that $f^{-1}(i)$ is an interval in $\{0,\ldots,k-1\}$ for all $i\in [\lambda]$. For $(E_0,\ldots,E_l)\in \mathfrak{E}_F$, let  $\mathfrak{E}_F((E_0,\ldots,E_l),k)$ denote the set of all weakly coloured racks of $F$ with $k$ layers whose associated coloured layer factorisation is $(E_0,\ldots,E_l)$.

Thus, we get that for $F\in \mathscr{M}(G)$ $$\sum_{k\geq 0}(-1)^k \beta_F(k)\binom{k}{i}\binom{\lambda-1}{i}=\sum_{(E_0,\ldots,E_l)\in \mathfrak{E}_F} \sum_{k\geq 0} (-1)^k |\mathfrak{E}_F((E_0,\ldots,E_l),k)|.$$

But, $$(-1)^k |\mathfrak{E}_F((E_0,\ldots,E_l),k)| = \sum_{x_0+\ldots+x_l=k} \prod_{i=0}^l(-1)^{x_i} \beta_{E_i}(x_i).$$ Here all the $x_i$'s are non-negative integers.
 
Thus, $$\sum_{k\geq 0}(-1)^k \beta_F(k)\binom{k}{i}\binom{\lambda-1}{i}=\sum_{(E_0,\ldots,E_l)\in \mathfrak{E}_F} \sum_{k\geq 0}\sum_{x_0+\ldots+x_l=k} \prod_{i=1}^l(-1)^{x_i} \beta_{E_i}(X_i).$$

The term on the right hand side becomes $$\sum_{(E_0,\ldots,E_l)\in \mathfrak{E}_F} \sum_{x_0,\ldots,x_l\geq 0} \prod_{i=0}^l(-1)^{x_i} \beta_{E_i}(x_i)=\sum_{(E_0,\ldots,E_l)\in \mathfrak{E}_F} \prod_{i=0}^l\left( \sum_{k\geq 0}(-1)^{k}\beta_{E_i}(k)\right).$$  

Using Equation~(\ref{inv}) we get that the right hand side is $(-1)^n|\mathfrak{C}_F|$, which is what we desire.\end{proof}

\section{Coefficients of Chromatic Polynomial and Heaps of Pieces }\label{coefficients}

We provide interpretation of the coefficients of the chromatic polynomial in terms of heaps of pieces. The interpretation was first provided by Greene and Zaslavsky in \citep{Greene1983}. Their proof involved Whitney numbers and hyperplane arrangements. 

We want to find the coefficient of the $r^{th}$ degree term $a_r$ in $\mathcal{X}_{G}(\lambda)$. From Equation~(\ref{chigeneral}) we get that the coefficient of $\lambda^r$ is 

\begin{equation}
a_r = \sum_{F \in\mathscr{M}(G)}\sum_{k\geq 0}\frac{\beta_F(k)}{k!}s(k,r)
\end{equation} where $s(k,r)$ is the signed Stirling number of the first kind.

Let $\Pi_G(k)$ denote the set of partitions of $V$ into $k$ independent sets of $G$. Let $\pi_G(k)=|\Pi_G(k)|$. It is not difficult to observe that for a fixed $k\in \mathbb{N}$, $$\sum_{F\in \mathscr{M}(G)}\frac{\beta_F(k)}{k!} = \pi_G(k).$$ Also, $s(k,r) = (-1)^{k-r}|s(k,r)|$ and $|s(k,r)|$ is the number of permutations in the symmetric group $S_k$ with $r$ cycles. Thus, we have $$a_r = \sum_{k\geq 0} (-1)^{k-r}\pi_G(k)|s(k,r)|.$$

This motivates us to define a keychain.\\

\begin{defi}[Keychain]Let $S$ be a partition of the vertex set of $G$ into independent sets. By an $r$-keychain on $S$ we mean an unordered partition of $S$ into $r$ non-empty subsets $C_0,\ldots,C_{r-1}$ each having a cyclic order. We call $C_i$ as the chains of the $r$-keychain. We call a $1$-keychain a chain as well. The size of an $r$-keychain on $S$ is $|S|$.

\end{defi}

Note that $\pi_G(k)|s(k,r)|$ denotes the number of $r$-keychains on $G$ for all possible $S$ of size $k$. If $\mathscr{K}_k$ denotes this number then we have that, $$a_r = (-1)^r\sum_{k\geq 0}(-1)^k\mathscr{K}_k.$$

Let $\delta_G(k)$ denote the number of chains of size $k$ for the underlying graph $G$. For $V'\subseteq V$ let $G(V')$ denote the subgraph of $G$ induced by the vertex set $V'$. We note that an $r$-keychain on a graph $G$ is a collection of $r$ chains on induced subgraphs of $G$. Thus, $$(-1)^k\mathscr{K}_k = \sum_{\substack{V = V_1\cup \ldots \cup V_r\\V_i\cap V_j = \phi\\}}\sum_{\substack{k_1,\ldots,k_r\\\sum{k_i} = k}} ((-1)^{k_1}\delta_{G(V_1)}(k_1))\cdot\ldots\cdot ((-1)^{k_r}\delta_{G(V_r)}(k_r)).$$

And hence,  
\begin{equation}\label{a_r}
a_r = (-1)^{r}\sum_{\substack{V = V_1\cup \ldots \cup V_r\\V_i\cap V_j = \phi}} \left(\sum_{k\geq 0}(-1)^k\delta_{G(V_1)}(k)\right)\cdot\ldots\cdot \left(\sum_{k\geq 0}(-1)^k\delta_{G(V_r)}(k)\right).
\end{equation}

We now assume that the vertices of the graph have a total order. Thus, a chain would correspond to a rack in which the largest vertex is on the lowermost layer. Similarly, an $r$-keychain would correspond to a collection of racks on induced subgraphs of $G$ where each rack has the largest piece on the lowermost layer. We call such racks as lower-special racks. Thus, we get that $\delta_G(k)$ is the number of lower-special racks with $k$ layers.

From Corollary~\ref{lowerspecial}, it is clear that we can restrict the heaps and racks lemma to lower-special racks. Thus, we apply heaps and racks involution to Equation~(\ref{a_r}) to get lexicographic racks on the $G(V_i)$ with the largest piece in the lowermost layer. The corresponding multilinear heap is an antipyramid. Thus, from the bijection we get that these antipyramids correspond to acyclic orientations in which there is a unique source with the source at the largest vertex. This is because if there were another source, then largest vertex couldn't be in the bottom layer. 

Thus, we get the following theorem:\\

\begin{thm}
Given an ordering on the set of vertices on a graph $G$, the number $a_r$ is the number of partitions of $G$ into $r$ induced subgraphs and where each induced subgraph has an acyclic orientation with a unique source at the largest vertex of the subgraph.\\
\end{thm}

\begin{cor}
$a_1$ is the number of acyclic orientations of $G$ with unique source at a selected vertex.
\end{cor}

\section{Acknowledgements}\label{acknowledgements}

I would like to thank Prof. Xavier Viennot and the Institute of Mathematical Sciences, Chennai, India for the fantastic courses on Combinatorics, especially the course on Heaps of Pieces offered in 2017 in which several of the topics appearing in this article were discussed and some of them were posed as research problems. I would like to thank Arun~Kumar~G., Prof. Amritanshu~Prasad, Prof. K.~N.~Raghavan, Prof. S.~Viswanath,  for their valuable suggestions and the several informal discussions that I had with them that led me to this article. I would like to thank Prof. Matthieu~Josuat-Vergès and Manjil~P.~Saikia for going over the drafts of the paper and their valuable suggestions. I would also like to thank the makers of the open source software Inkscape, using which I made the figures. Finally, I would like to thank LabEx Bezout for funding my M2 studies at LIGM, UPEM.
\bibliography{research}

\bibliographystyle{dinat}

\section*{Appendix: Proofs from section 5}

\begin{proof}[Proof of Lemma~\ref{algo}] We need to show that for a rack $T$ of $F$, the rack obtained after performing the algorithm on $T$ is contained in $R(F)$. Let the rack be $T = (L_0,\ldots, L_{k-1})$ where the $L_i$ are the layers and as each $L_i$ can be thought of as a word in letters, all of which commute with each other, $L_0\odot\ldots\odot L_{k-1} = F$.

If $T$ does not have a transfer piece then $T$ is fixed by the algorithm and hence, the lemma trivially follows. 

If $T$ has a transfer piece then we assume that the transfer piece $p$ is in the layer $L_j$ and we have the following two cases:

\underline{Case 1}: When the piece $p$ is not lonely. Let $L_j'$ denote the layer $L_j$ without the piece $p$. As all pieces on $L_j$ commute with each other we have that $L_j' \odot p = L_j$ and hence we get that $$L_0\odot\ldots L_{j-1}\odot L_{j}' \odot p \odot L_{j+1}\odot \ldots \odot L_{k-1}  = F .$$ Thus, $(L_0,\ldots,L_{j-1}, L_{j}' , p ,L_{j+1}, \ldots , L_{k-1}) \in R(F)$  which is also the rack that we get after applying the algorithm to $T$.

\underline{Case 2}: When the piece $p$ is lonely. Thus, we  get that $L_j = p$. We need to show that $L_{j-1}\odot p = p\odot L_{j-1}$, i.e., $p$ commutes with all pieces in $L_{j-1}$. If not, then there is a  piece $q\in L_{j-1}$ such that that $p\odot q \neq q\odot p$. Thus, in lexicographic order $q<p$. As $q$ is not the transfer piece it must be lonely and the label of $q$ is $j-1$. Thus, the label of $p$ is greater than or equal to $j$. It cannot be $j$ as $p$ is the transfer piece. Let $r$ be the piece with the label $j$. As $r$ is not in the $j^{th} $ layer, there is a piece less than $p$ which is not in the layer as its label. Thus, we get a contradiction. 

Thus, we get that $p$ commutes with $L_{j-1}$ and hence we get that $(L_0,\ldots,L_{j-2},L_{j-1}\odot p, L_{j+1},\ldots, L_{k-1})\in R(F)$.

Hence the lemma holds.\end{proof}

We now proof the main lemma in this article.

\begin{proof}[Proof of the Heaps and Racks lemma] %We need to check that $$\sum_{T\in R(F)} (-1)^{\#(T)} = (-1)^{|F|}$$

Let the rack be $T = (L_0,\ldots, L_{k-1})$. Let $S$ be the transfer set of $T$ and let $p$ be the transfer piece contained in the layer $L_{j}$.

We now show that $f(f(T)) = T $ for $T$ when $T$ is not the lexicographic rack. We show that the transfer piece of $T$ and $f(T)$ is the same piece. We have following two cases:

\underline{Case 1}: When $p$ is not lonely. Let $L_{j}'$ be the layer $L_j$ after removing $p$. Thus, $f(T) = (L_0,\ldots,L_{j-1}, L_{j}' , p ,L_{j+1}, \ldots , L_{k-1})$. We show that the transfer piece of $f(T)$ is $p$. If not then there is piece $q$ different from $p$ which is the transfer piece of $f(T)$. Thus, $q$ should have label $i$ less than label of $p$. As it was not the transfer piece of $T$, it would have been lonely and is still lonely by assumption of this case. It was also in the layer $i$. But its layer number in $f(T)$ is different from its layer number in $T$. Thus, $i>j$. 

Let the label of $p$ be $l$. Let the $j^{th}$ piece be $r$. As the layer $L_j$ has $p$, $j<l$ is not true else $p$ cannot become the transfer piece of $T$. Thus, we get that $j\geq l$. Thus, get that $j\geq l > i$ which contradicts $i>j$.

Thus, $p$ is the transfer piece of $f(T)$ as well and hence $f(f(T)) = T$.

\underline{Case 2}: When $p$ is lonely. We get that $f(T) = (L_0,\ldots,L_{j-2},L_{j-1}\odot p, L_{j+1},\ldots, L_{k-1})$. Similar to case 1 we need to show that $p$ is the transfer piece of $f(T)$ as well. If not then let the transfer piece be $q$ which is different from $p$. Let $q\in L_i$. $q$ has to be a lonely piece in $T$ as otherwise it would be in $S$ and thus we would have a piece in $S$ smaller than $p$.

If $i<j-1$ then the layer number of $q$ is same in both $T$ and $f(T)$ and hence $q\in S$. Thus, $q$ was a piece in $S$ smaller than $p$ which contradicts the fact that $p$ was the transfer piece.

If $i=j-1$, then label of $q$ is $j-1$. Thus, as $q$ is the transfer piece of $f(T)$, all pieces less than $q$ are not in $S$, i.e., for all $0\leq h<j-1$, $L_h$ has only the $h^{th}$ piece. Now, $p$ cannot be the $j^{th}$ piece as otherwise it would not have been in $S$. Let the $j^{th}$ piece be $r$. Clearly, $r\in S$ which contradicts that $p$ was the transfer piece of $T$.

If $i>j$, then the $(j-1)^{st}$ layer of $f(T)$ is not a singleton layer and hence the $(j-1)^{st}$ piece is in the transfer set of $f(T)$. This contradicts that $q$ is the transfer piece of $f(T)$

Thus, in this case as well we have that $f(f(T)) = T$.\end{proof}

\end{document}